\documentclass[reqno,11pt]{amsart}
\usepackage[utf8]{inputenc}

\calclayout
\usepackage{amsthm,amsfonts,amstext,amssymb,mathrsfs,amsmath,latexsym,mathtools,mathabx}
\usepackage{enumerate}
\usepackage{bm,accents}
\usepackage[abs]{overpic}
\usepackage{mathrsfs}
\usepackage{hyperref}
\usepackage{esint} 
\usepackage{color}
\usepackage[colorinlistoftodos,prependcaption,textsize=small]{todonotes}
\usepackage{tikz}

\theoremstyle{plain}
\newtheorem{theorem}{Theorem}[section]
 
\newtheorem{lem}[theorem]{Lemma}
\newtheorem{corollary}[theorem]{Corollary} 

\theoremstyle{definition}

\theoremstyle{remark}
\newtheorem{remark}[theorem]{Remark}
\numberwithin{equation}{section}

\renewcommand{\Re}{\mathop{\rm Re}}
\renewcommand{\Im}{\mathop{\rm Im}}

\DeclareMathOperator{\supp}{supp}

\newcommand{\N}{\mathbb{N}}
\newcommand{\C}{\mathbb{C}}

\newcommand{\R}{\mathbb{R}}

\definecolor{apricot}{rgb}{0.98, 0.81, 0.69}

\definecolor{greenpie}{rgb}{0.69, 0.95, 0.76}

\title[Weighted equilibrium in a field of a uniform charge]{Weighted equilibrium in a field of a uniform charge of an interval}

\author[J.~Kessinger]{James Kessinger$^*$}
\thanks{$^*$ Corresponding author.}

\address[JK]{Department of Mathematics, Baylor University, Waco, TX 76706, USA}

\email{James\_Kessinger1@baylor.edu}

\author[A.~Mart\'{\i}nez-Finkelshtein]{Andrei Mart\'{\i}nez-Finkelshtein}

\address[AMF]{Department of Mathematics, Baylor University, Waco, TX 76706, USA, and Department of Mathematics, University of Almer\'{\i}a, Almer\'{\i}a, Spain}

\email{A\_Martinez-Finkelshtein@baylor.edu}

\date{\today}



\begin{document}

\begin{abstract}
    
We study the logarithmic equilibrium problem on the interval $[-1,1]$ in the presence of an external field generated by a uniform background charge supported on the same interval. For a real parameter $\tau$, the external field is taken to be $\tau$ times the logarithmic potential of the unit Lebesgue measure, and for all values of $\tau$ we determine explicitly the unique equilibrium measure $\mu_\tau$, its support, its Cauchy transform, its logarithmic potential (when a closed expression is available), and the equilibrium constant.

We show that the model exhibits three distinct regimes separated by critical values of $\tau$. For sufficiently negative $\tau$, the equilibrium support is a single symmetric subinterval strictly contained in $[-1,1]$. For an intermediate range of parameters, the support coincides with the full interval, and the equilibrium measure is an explicit linear combination of the Robin distribution and the Lebesgue measure. For large positive $\tau$, the support becomes disconnected and consists of two symmetric outer intervals. In each regime, we find the equilibrium measure, its Cauchy transform, its potential (when a closed expression is available), and the equilibrium constant, using complex-analytic methods and singular integral techniques. These results yield a complete picture of how the support topology and the equilibrium density/constant evolve as $\tau$ varies, including the transitions between one-cut, full-support, and two-cut configurations.
\end{abstract}

\maketitle

\section{Introduction} \label{sec:intro}

The classical equilibrium problem in potential theory asks for the probability measure on a compact set that minimizes the logarithmic energy. This fundamental problem has deep connections to approximation theory, random matrix theory, orthogonal polynomials, and other areas of mathematics and physics.

For a (signed) Borel measure $\sigma$ on $[-1,1]$, we denote by 
\begin{equation*}\label{Def_Log_Potential}
    V^\sigma(z):=\int\log\frac{1}{|x-z|}\, d\sigma(x)
\end{equation*}
its logarithmic potential. A well-known fact is that 
$$
d\eta(x):=\frac{1}{\pi \sqrt{1-x^2}}\, dx, \quad x\in (-1,1),
$$
is the minimizer, among all probability (unit and positive) measures on $[-1,1]$, of the energy
$$
\mathcal I[\sigma]:= \int V^\sigma \, d\sigma .
$$
The measure $\eta$, known as the equilibrium (or Robin) distribution of $[-1,1]$, is characterized by the property that its potential takes a constant value on the interval.

In the presence of an external field (or background potential) $\varphi:[-1,1]\to \R\cup \{\pm \infty\}$, the equilibrium measure $\mu$ is the probability measure on $[-1,1]$ minimizing the total energy
$$
\mathcal I_\varphi[\sigma]:= \int \left( V^\sigma +2\varphi  \right)\, d\sigma =\mathcal I[\sigma]+ 2 \int \varphi\, d\sigma
$$
among all probability  measures on $[-1,1]$, see e.g.~\cite{Saff:97}. If $\varphi$ is smooth on $(-1,1)$, then the measure $\mu$ is characterized by the following equilibrium conditions: there exists a constant $\omega\in \R$ such that
\begin{equation*}
   V^\mu(x)  +\varphi(x) \begin{cases}
       = \omega, & x\in \supp (\mu), \\
    \ge \omega, & x\in [-1,1].
   \end{cases}
\end{equation*}
The determination of $\supp (\mu)$, which may be a proper subset of  $[-1,1]$, is the central challenge.

In this paper, we consider the case where the external field $\varphi$ is itself a logarithmic potential, specifically, a multiple of the potential generated by the Lebesgue measure on the interval $[-1,1]$. To be more precise, let $\lambda$ be the absolutely continuous measure $\frac{1}{2}dx$ on $[-1,1]$. Given a parameter $\tau\in \R$, we are interested in the (unique) probability measure $\mu_\tau$ on $[-1,1]$ satisfying the equilibrium conditions
\begin{equation}
    \label{equilibrium}
   V^{\mu_\tau}(x)  +\tau V^{\lambda}(x) \begin{cases}
       = \omega_\tau, & x\in \supp (\mu_\tau), \\
    \ge \omega_\tau, & x\in [-1,1].
   \end{cases}
\end{equation}
This model has a natural physical interpretation: the equilibrium measure can be viewed as the distribution of charge that minimizes energy in the presence of a uniform background charge distributed along $[-1,1]$, where $\tau$ controls the strength and sign of the interaction.

Our main contribution is a complete description of $\mu_\tau$, its potential $V^{\mu_\tau}$ (when a closed analytic expression is available), and the equilibrium constant $\omega_\tau$ for all values of $\tau$. A central tool in our analysis is the Cauchy transform of $\mu_\tau$,
\begin{equation*}
    \label{CauchyTransf}
       C^{\mu_\tau}(z):=\int\frac{1}{z-x}\ d{\mu_\tau}(x),\quad z\in\mathbb{C}\setminus \supp({\mu_\tau}),
\end{equation*}
which encodes the equilibrium density via the Sokhotski–-Plemelj formula and is the primary object we determine in each regime.

We establish that there are three different regimes:
\begin{description}
    \item[Attractive regime ($\tau<-1$)] The support contracts to a symmetric proper subinterval,  $\supp \left(\mu_\tau\right)=\left[-\beta_\tau, \beta_\tau\right]$, where $0<\beta_\tau<1$ is explicitly determined by $\tau$. The measure is absolutely continuous with respect to Lebesgue measure, and its density involves inverse trigonometric functions. The external field is strictly convex, causing the equilibrium measure to concentrate away from the endpoints.
    
    \item[Intermediate regime ($-1\le \tau\le  2/(\pi-2)$)] The support is the full interval $\supp\left(\mu_\tau\right)= [-1,1]$, and the equilibrium measure has the remarkably simple form $\mu_\tau=(1+\tau) \eta-\tau \lambda$, an explicit linear combination of the Robin distribution and the Lebesgue measure. This regime includes both the classical case ($\tau=0$, where $\mu_0=\eta$) and cases with both attractive and repulsive components. 
    
    \item[Repulsive regime ($\tau>2/(\pi-2)$)] The support becomes disconnected: now $\supp (\mu_\mathbf{\tau})= [-1,-\beta_\tau] \cup [\beta _\tau, 1]$, where $\beta_\tau \in(0,1)$ is characterized as the solution of an equation involving the complete elliptic integral of the second kind. The external field is concave, pushing the equilibrium measure toward the endpoints of $[-1,1]$. 
\end{description}
The connectedness of the support for $\tau<0$ follows immediately from the convexity of the external field $\varphi$ (cf.~Figure~\ref{fig:ext_field}), a consequence of a general theorem in logarithmic potential theory (see e.g.~\cite[Theorem IV.1.10]{Saff:97}). The critical value $\tau=2 /(\pi-2) \approx 1.753$ marks the transition from connected to disconnected support.

\begin{figure}[h!]
    \centering
        \includegraphics[width=0.5\linewidth]{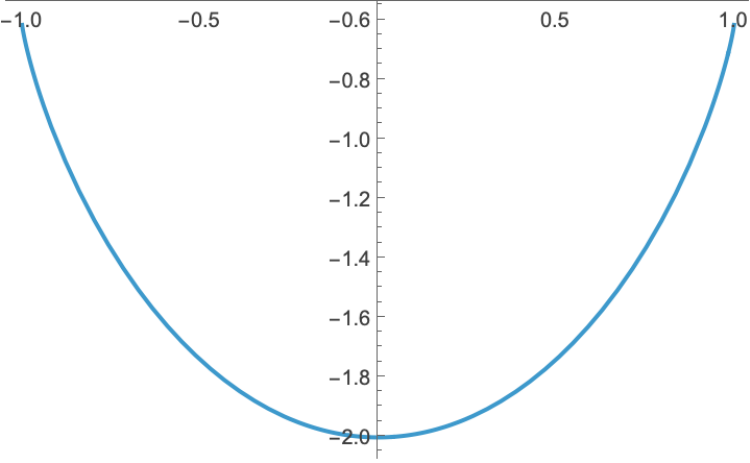}
        \caption{The external field  for $\tau=-2$.}
        \label{fig:ext_field}
\end{figure}

A primary motivation for this model comes from the theory of orthogonal polynomials. Consider polynomials $p_n(x)=\kappa_n x^n+\cdots$ orthogonal with respect to the varying weight $e^{-2 n \tau V^\lambda(x)} d x$ on $[-1,1]$, that is,
$$
\int_{-1}^1 p_n(x) p_m(x) e^{-2 n \tau V^\lambda(x)} d x=\delta_{n m}.
$$
By the general theory of weighted approximation (see e.g.~\cite{Saff:97}), the normalized zero-counting measures of $p_n$ converge in a weak-* sense to the equilibrium measure $\mu_\tau$ as $n \rightarrow \infty$. Thus, the three regimes identified in this paper correspond to qualitatively different asymptotic zero distributions: for $\tau<-1$ the zeros cluster on a proper subinterval and stay away from $\pm 1$; for intermediate $\tau$ they fill $[-1,1]$ with a density that interpolates between the arcsine law and the uniform distribution; and for $\tau>2 /(\pi-2)$ they split into two symmetric groups migrating toward the endpoints. This bifurcation of the support is a manifestation of a phase transition of the type well-studied in random matrix theory \cite{BleherIts, DeiftKriecherbauer}, where analogous transitions -- called a ``birth of a new cut'' or a ``closing of a gap" -- arise as a parameter in the matrix potential crosses a critical value and the limiting spectral measure changes topology. The critical value $\tau=2 /(\pi-2)$ is the precise threshold at which such a topological change occurs in our model. At the special value $\tau=0$ there is no external field, and the equilibrium measure is the classical arcsine distribution, recovering the zero distribution of Chebyshev polynomials; at $\tau=-1$ it reduces to the uniform measure $\lambda$ itself.

Our analysis relies on complex analytic techniques, including the study of Cauchy transforms, boundary value relations, and singular integral representations. In each regime, we derive explicit formulas or integral representations for the equilibrium density and provide closed-form expressions for the associated Cauchy transform. Together, these results yield a unified description of the equilibrium structure across the full parameter range and illustrate how external forcing induces topological changes in the support.

The paper is organized according to the three regimes. We first treat the full-support case (Section \ref{sec:case1}), where explicit formulas follow from linear combinations of known equilibrium measures. We then analyze the one-cut regime (Section \ref{sec:case2}) using analytic continuation and boundary value methods on a suitable Riemann surface. Finally, in Section \ref{sec:case3} we study the disconnected regime via singular integral techniques and elliptic-type representations. In this case, the equilibrium constant $ \omega_\tau$ can be 
obtained via a convergent series whose coefficients satisfy a three-term recurrence relation. 
Each section culminates in explicit expressions for the equilibrium measure and its analytic transforms.

\underline{Notational convention:} Along the paper, we denote by $\sqrt{\cdot}$ the positive square root of non-negative real values, while $\left(\cdot\right)^{1/2}$ stands for the principal branch of the complex square root, which coincides with $\sqrt{\cdot}$ for the positive argument. Furthermore, unless stated otherwise, by $\arg(z)$ we understand the principal branch of the function in $\C \setminus (-\infty, 0]$, such that $\arg(z)=0$ for $z>0$. In consequence, $\log(z)$ stands for the principal branch in $\C \setminus (-\infty, 0]$, $\log(z)=\ln|z|+i \arg(z)$.

\section{Case of the full support: \texorpdfstring{$-1\le \tau\le 2/(\pi-2)$}{-1\le \tau\le 2/(\pi-2)}} \label{sec:case1}

We begin with the case when $\supp(\mu_\tau)=[-1,1]$. Notice that for $\tau=0$, $\mu_0=\eta$, the Robin distribution, so that
$$
d\mu_0(x):=\frac{1}{\pi \sqrt{1-x^2}}\, dx, \quad x\in (-1,1);
$$
it is also known (and easy to compute) that in this situation,
$$
V^{\mu_0}(x) =\omega_0=  \log(2), \quad x\in [-1,1], 
$$
and
$$
V^{\mu_0}(z) = -\Re \log\left(\frac{\Phi(z)}{2} \right),
$$
where
\begin{equation}
    \label{defPhi}
    \Phi(z):= z+\left(z^2-1 \right)^{1/2}, \quad \Phi:\C\setminus [-1,1] \longrightarrow \left\{ w\in \C:\, |w|>1\right\},
\end{equation}
is the inverse of the Joukowski map.

\begin{theorem}\label{theorem_case_-1<=tau<=2/(pi-2)}
    For $-1\leq\tau\leq\frac{2}{\pi-2}$, $\mu_\tau$ is an absolutely continuous measure on $[-1,1]$ given by
    \begin{equation}\label{mu_when_-1<=tau_<=/(pi-2)}
        d\mu_\tau(x)=\left(\frac{1+\tau}{\pi\sqrt{1-x^2}}-\frac{\tau}{2}\right)\ dx,
    \end{equation}
so that
\begin{equation}\label{Cauchy_Transform_mu_when_-1<=tau_<=/(pi-2)}
        C^{\mu_\tau}(z)= \frac{1+\tau}{\left( z^2-1\right)^{1/2}}-\frac{\tau}{2}\log\left(\frac{z+1}{z-1}\right).
    \end{equation}
Furthermore, for $z\in \C\setminus [-1,1]$,  
\begin{equation*}
\begin{split}
        V^{\mu_\tau}(z)= & -(1+\tau)\log\left|\frac{z+\left( z^2-1\right)^{1/2}}{2}\right|-\tau  \\
        & -\frac{\tau}{2}\left(\operatorname{Re}(z)\log\left|\frac{z-1}{z+1}\right|+\operatorname{Im}(z)\arg\left(\frac{z+1}{z-1}\right)-\log|z^2-1|\right),
\end{split}
    \end{equation*}
and
\begin{equation}\label{Log_Potential_mu_when_-1<=tau_<=/(pi-2)_[-1,1]}
    V^{\mu_\tau}(x)=\frac{\tau}{2} \left[ (1+x)\log(1+x)+ (1-x)\log(1-x) -2 \right] + \omega_\tau, \quad x\in   [-1,1],
\end{equation}
where the equilibrium constant $\omega_\tau$ in \eqref{equilibrium} is  
\begin{equation}\label{equilibrium_constant_-1<=tau<=2/(pi-2)}
    \omega_\tau=(1+\tau)\log(2).
\end{equation}
\end{theorem}
\begin{proof}
    The measure $\mu_\tau$ in \eqref{mu_when_-1<=tau_<=/(pi-2)} is \begin{equation*}\label{Form_of_mu_when_-1<=tau_<=/(pi-2)}
        \mu_\tau=(1+\tau)\eta-\tau\lambda.
    \end{equation*}
    Since $\eta$ and $\lambda$ are probability measures, it follows that \eqref{mu_when_-1<=tau_<=/(pi-2)} is a unit measure. By assumption, $\tau+1>0$, so that the minimum of the density in the right-hand side of \eqref{mu_when_-1<=tau_<=/(pi-2)}, attained for $x=0$, is
    $$
   \frac{1+\tau}{\pi }-\frac{\tau}{2},
    $$
    which is nonnegative due to the upper bound on $\tau$. In consequence, \eqref{mu_when_-1<=tau_<=/(pi-2)} is a positive (and hence a probability) measure on $[-1,1]$.

    Furthermore, for $x\in[-1,1]$,  
    \begin{align*}
      V^{\mu_\tau}(x)+\varphi(x)=   V^{(1+\tau)\eta-\tau\lambda}(x)+\tau V^{\lambda}(x)=(1+\tau)V^\eta(x)=(1+\tau)\log(2).
    \end{align*}
    This proves that \eqref{mu_when_-1<=tau_<=/(pi-2)} is the equilibrium measure and gives us the expression in \eqref{equilibrium_constant_-1<=tau<=2/(pi-2)}. 
    
    Finally, \eqref{Cauchy_Transform_mu_when_-1<=tau_<=/(pi-2)}--\eqref{Log_Potential_mu_when_-1<=tau_<=/(pi-2)_[-1,1]} follow from the known explicit expressions for the Cauchy transform for $\lambda$,
    \begin{equation*} \label{CauchyforLebesgue}
        C^{\lambda}(z)= \frac{1}{2}\log\left(\frac{z+1}{z-1}\right)=\frac{1}{2}\log\left(1+\frac{2}{z-1}\right),
    \end{equation*}
    its primitive,
    \begin{equation} \label{GforLebesgue}
    \begin{split}
        g^{\lambda}(z)&=\int^z C^{\lambda}(y)\, dy= \frac{1}{2} (z+1) \log (z+1)-\frac{1}{2} (z-1) \log (z-1)-1 \\
        &= z \, C^{\lambda}(z)+\frac{1}{2}  \log (z^2-1) -1,
         \end{split}
    \end{equation}
    normalized to have the expansion
    $$
    g^{\lambda}(z)=\log(z)+\mathcal O(z^{-1}), \quad z\to\infty, 
    $$
    and the logarithmic potential of $\lambda$,
    $$
    V^{\lambda}(z)=-\Re g^{\lambda}(z),
    $$
    as well as of the analogous functions of $\eta$.
\end{proof}

\section{Connected support (one-cut case): \texorpdfstring{$\tau<-1$}{\tau<-1}} \label{sec:case2}

\begin{theorem}\label{theorem_case_tau<-1}
    For $\tau<-1$, $\mu_\tau$ is an absolutely continuous measure with 
    \begin{equation}\label{supp_mu_tau<-1}
        \supp(\mu_\tau)=\left[-\beta_\tau,\beta_\tau\right], \quad \beta_\tau= \sqrt{1-\left(\frac{1+\tau}{\tau}\right)^2 }\in (0,1),
    \end{equation}
    given by 
    \begin{equation}\label{mu_tau<-1_expression_1}
        d\mu_\tau(x)=-\frac{\tau}{\pi}\left(\frac{\pi}{2}-\arctan \left( \sqrt{\frac{1-\beta_\tau^2}{\beta_\tau^2-x^2}}\right)\right)\, dx,
    \end{equation}
so that,
\begin{equation}\label{Cauchy_Transform_tau<-1_expression_2}
        C^{\mu_\tau}(z)
        = \tau\log\left(\frac{(z^2-\beta_\tau^2)^{1/2}+\sqrt{1-\beta_\tau^2}}{z+1}\right), \quad z\in \C\setminus\left[-\beta_\tau,\beta_\tau\right].
    \end{equation}
    Furthermore, for $z\in \C\setminus [-\beta_\tau,\beta_\tau]$,  
    $$
     V^{\mu_\tau}(z)=- \Re g^{\mu_\tau}(z), 
    $$
    where
    \begin{equation}\label{g_function_tau<-1}
    \begin{split}
        g^{\mu_\tau}(z)&=zC^{\mu_\tau}(z)+(1+\tau)\log\left(\frac{\beta_\tau}{2}\Phi\left(\frac{z}{\beta_\tau}\right)\right)-\tau\log\left(\frac{(z^2-\beta_\tau^2)^{1/2}+z\sqrt{1-\beta_\tau^2}}{1+\sqrt{1-\beta_\tau^2}}\right)-1
    \end{split}
    \end{equation}
    and $\Phi$ is defined as in \eqref{defPhi}. In particular, for $z\in\mathbb{C}\setminus[-\beta_\tau,\beta_\tau]$,
    \begin{equation}\label{log_potential_tau<-1}
    \begin{split}
        V^{\mu_\tau}(z)&=-\operatorname{Re}(zC^{\mu_\tau}(z))-(1+\tau)\log\left|\frac{\beta_\tau}{2}\Phi\left(\frac{z}{\beta_\tau}\right)\right|+\tau\log\left|\frac{(z^2-\beta_\tau^2)^{1/2}+z\sqrt{1-\beta_\tau^2}}{1+\sqrt{1-\beta_\tau^2}}\right|+1.
    \end{split}
    \end{equation}
   On the support $\supp(\mu_\tau)=\left[-\beta_\tau,\beta_\tau\right]$, $V^{\mu_\tau}$ is given by the same expression \eqref{Log_Potential_mu_when_-1<=tau_<=/(pi-2)_[-1,1]}, but now the equilibrium constant $\omega_\tau$ is  
\begin{equation}\label{equilibrium_constant_tau<-1}
        \omega_\tau=(1+\tau)\log(2)-\log(\beta_\tau)+1+\tau-\tau\log\left(1+\sqrt{1-\beta_\tau^2}\right).
    \end{equation}
\end{theorem}
For the plot of the density of the equilibrium measure $\mu_\tau$ for $\tau=-2$, see   Figure~\ref{fig:density}.

\begin{figure}[h!]
    \centering
        \includegraphics[width=0.5\linewidth]{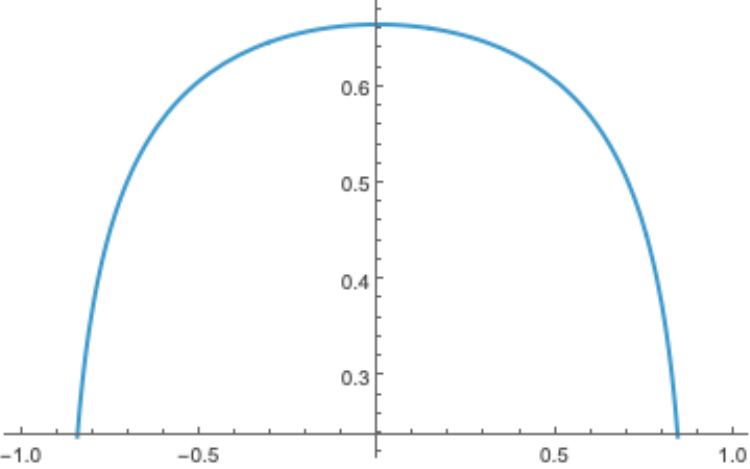}
        \caption{The density $\mu_{\tau}'$ for $\tau=-2$, with $\beta_{-2}=\sqrt{3}/2$.}
        \label{fig:density}
\end{figure}
\begin{remark}
\begin{enumerate} 
    \item[\textit{(i)}] Comparing the expressions \eqref{equilibrium_constant_-1<=tau<=2/(pi-2)} and \eqref{equilibrium_constant_tau<-1} we see that the equilibrium constant $\omega_\tau$ is continuous across $\tau=-1$:
$$
\lim_{\tau \to -1^-} \omega_\tau = \lim_{\tau \to -1^+} \omega_\tau=0.
$$
This is consistent with the fact that
$$
\mu_\tau \big|_{\tau=-1^-}=\mu_\tau \big|_{\tau=-1^+}=\lambda. 
$$
\item[\textit{(ii)}]    The density $\mu_\tau'$ exhibits the square-root asymptotics at the end-points of the support, corresponding to the standard ``soft edge'' behavior:
    \[\mu_\tau'(x)=\frac{\tau}{\pi}\sum_{n=1}^\infty\frac{(-1)^{n+1}}{n}\left(\frac{\beta_\tau^2-x^2}{1-\beta_\tau^2}\right)^{\frac{2n-1}{2}}=\frac{\tau}{\pi}\sqrt{\frac{\beta_\tau^2-x^2}{1-\beta_\tau^2}}\left(1+\mathcal{O}(|x\mp\beta_\tau|)\right),\quad x\to\pm\beta_\tau.\]
    \end{enumerate}
\end{remark}
\begin{proof}
    For $\tau<0$, $\tau V^\lambda$ is downward-convex on $[-1,1]$, see Figure \ref{fig:ext_field}, and by \cite[Theorem IV.1.10, (b)]{Saff:97},  $\supp(\mu_\tau)$ is a sub-interval on $[-1,1]$. Since $\varphi$ is an even function, $\supp(\mu_\tau)$ is symmetric with respect to the origin, so that $\supp(\mu_\tau)=[-\beta_\tau,\beta_\tau]$, where the value of $\beta_\tau\in[0,1]$ is to be determined. 
    
    Function $f:\mathbb{C}\setminus[-1,1]\to\mathbb{C}$ defined by 
    $$
    f(z):=C^{\mu_\tau}(z)+\tau C^\lambda(z)=\int\frac{ d{\mu_\tau}(x)}{z-x}+ \tau \int\frac{ d{\lambda}(x)}{z-x}
    $$
    is holomorphic (analytic and single-valued) on its domain, and   
\begin{equation}\label{f_tau<-1_infinity_property}
        f(z)=\frac{1+\tau}{z}+\mathcal O(z^{-2}), \quad z\to\infty.
    \end{equation}
Since by definition, $f$ is real-valued on $\R\setminus [-1,1]$, the symmetry principle implies that
\begin{equation}
    \label{symmetryprinciple}
    f(\overline{z}) =\overline{f(z)}, \quad z\in \C\setminus [-1,1].
\end{equation}
Our goal is to obtain an explicit expression for $f$ using the properties enumerated next, and in this way, to recover $\mu_\tau$. 
    
By \cite[Chapter IV, in particular, Theorems IV.2.2 and IV.3.1]{Saff:97}, the density of $\mu_\tau$ is locally Lipschitz-continuous on its support, and thus, $f$ has continuous non-tangential boundary values $f_\pm$ from the upper/lower half-plane on $(-1,1)$. Since $C^{\mu_\tau}$ is analytic in $\C \setminus [-\beta_\tau,\beta_\tau]$, by Sokhotski–-Plemelj’s formula,   
    \begin{equation}
        f_+(x)-f_-(x)=-\tau\pi i,  \quad  x\in(-1,-\beta_\tau)\cup(\beta_\tau,1).
        \label{f_tau<-1_limit_subtract_property}
    \end{equation}
 
   The total potential 
    \begin{equation}
        \label{totalpotential}
    V^{\mu_\tau}(x)+\tau V^\lambda(x)=-\Re \int^z f(s) ds
     \end{equation}
    is continuous on $(-1,1)$ (see \cite[Section IV.2]{Saff:97}), and equilibrium conditions \eqref{equilibrium} imply that it
    is constant on $[-\beta_\tau,\beta_\tau]$. Combining it with \eqref{symmetryprinciple}, we conclude that 
    \begin{equation}
        f_+(x)+f_-(x)=2\Re f_\pm(x)=0, \quad  x\in(-\beta_\tau,\beta_\tau).
        \label{f_tau<-1_limit_sum_property}
    \end{equation}
 Motivated by the inequalities in \eqref{equilibrium}, we make an ansatz that the total potential is strictly increasing (respectively, decreasing) on $(\beta_\tau,1)$ (respectively, on $(-1,-\beta_\tau)$). By formula \eqref{totalpotential}, this means that
    \begin{equation}
        \Re f_\pm(x)\geq 0 \text{ on } (-1,-\beta_\tau), \quad  \text{and} \quad \Re f_\pm(x)\leq0 \text{ on } (\beta_\tau,1).
        \label{f_tau<-1_Re_f_sign_property}
    \end{equation}
    
To find an explicit expression for $f$, we create a $2$-sheeted Riemann surface $\mathcal{R}$ of genus 0 by taking two copies ($\mathcal{R}_1$ and $\mathcal{R}_2$) of the Riemann sphere cut along $[-\beta_\tau,\beta_\tau]$ and gluing them along the cut with opposite orientation, see Figure \ref{fig:Riemann_Surface}. We will use $\zeta$ for the local coordinate on $\mathcal R$. Denote by $L=[-1,1]\backslash[-\beta_\tau,\beta_\tau]$, and by $\ell_1\subset \mathcal{R}_1$ and $\ell_2 \subset\mathcal{R}_2$, the lift of this set to the sheets $\mathcal R_1$ and $\mathcal R_2$, respectively, with $\ell:=\ell_1 \cup \ell_2$.

    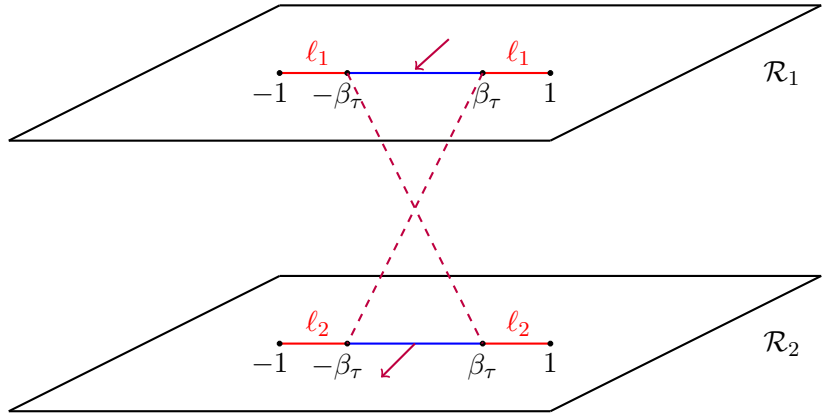
\begin{figure}[ht]
        \centering
        \begin{tikzpicture}[scale=0.9]
        \draw[black, thick] (-2,8) -- (6,8);
        \draw[black, thick] (-6,6) -- (2,6);
        \draw[black, thick] (-2,8) -- (-6,6);
        \draw[black, thick] (2,6) -- (6,8);

        \draw[black, thick] (-2,4) -- (6,4);
        \draw[black, thick] (-6,2) -- (2,2);
        \draw[black, thick] (-2,4) -- (-6,2);
        \draw[black, thick] (2,2) -- (6,4);

        \draw[blue, thick] (-1,7) -- (1,7);
        \draw[red, thick] (-2,7) -- (-1,7);
        \draw[red, thick] (1,7) -- (2,7);

        \draw[blue, thick] (-1,3) -- (1,3);
        \draw[red, thick] (-2,3) -- (-1,3);
        \draw[red, thick] (1,3) -- (2,3);

        \filldraw[black] (-2.,7) circle (1pt) node[anchor=north]{$-1\ \ $};
        \filldraw[black] (-1,7) circle (1pt) node[anchor=north]{$-\beta_\tau\ \ $};
        \filldraw[black] (1,7) circle (1pt) node[anchor=north]{$\ \beta_\tau$};
        \filldraw[black] (2,7) circle (1pt) node[anchor=north]{$1$};
        \filldraw[red] (-1.1,7.3) circle (0pt) node[anchor=east]{$\ell_1$};
        \filldraw[red] (1.2,7.3) circle (0pt) node[anchor=west]{$\ell_1$};

        \filldraw[black] (-2.,3) circle (1pt) node[anchor=north]{$-1\ \ $};
        \filldraw[black] (-1,3) circle (1pt) node[anchor=north]{$-\beta_\tau\ \ $};
        \filldraw[black] (1,3) circle (1pt) node[anchor=north]{$\beta_\tau$};
        \filldraw[black] (2,3) circle (1pt) node[anchor=north]{$1$};
        \filldraw[red] (-1.1,3.3) circle (0pt) node[anchor=east]{$\ell_2$};
        \filldraw[red] (1.2,3.3) circle (0pt) node[anchor=west]{$\ell_2$};

        \draw[purple, dashed, thick] (-1,7) -- (1,3);
        \draw[purple, dashed, thick] (1,7) -- (-1,3);
        \draw[purple, thick, ->] (0.5,7.5) -- (0,7.05);
        \draw[purple, thick, ->] (0,3) -- (-0.5,2.5);

        \filldraw[black] (5,7) circle (0pt) node[anchor=west]{$\mathcal{R}_1$};
        \filldraw[black] (5,3) circle (0pt) node[anchor=west]{$\mathcal{R}_2$};
    \end{tikzpicture}
        \caption{The Riemann surface $\mathcal{R}.$}
        \label{fig:Riemann_Surface}
    \end{figure}
    
    We define on $\mathcal R\setminus \ell$ the complex-valued function
    \begin{equation*}\label{F_tau<-1}
      \mathcal  F(\zeta)=\begin{cases}
            f(\zeta), &  \zeta\in\mathcal{R}_1 \setminus \ell_1, \\
                 -f(\zeta), &   \zeta\in\mathcal{R}_2 \setminus \ell_2.
              \end{cases}
    \end{equation*}
    Condition \eqref{f_tau<-1_limit_sum_property} guarantees that $\mathcal F$ is also analytic across $(-\beta_\tau,\beta_\tau)$; thus, it is analytic on $\mathcal{R}\backslash \ell$. Furthermore, condition \eqref{f_tau<-1_limit_subtract_property} lifts to $\mathcal R$ as 
    \begin{equation}
        \label{bdrycondonR}
    (\mathcal F_+-\mathcal F_-)(\zeta)=-\tau\pi i, \quad \zeta\in \ell_1, \qquad  (\mathcal F_+-\mathcal F_-)(\zeta)=\tau\pi i, \quad \zeta\in \ell_2.
    \end{equation}

   With the notation \eqref{defPhi}, define
    \begin{equation*}\label{varphi}
         \phi(\zeta):=\begin{cases}
             \Phi(\zeta/\beta_\tau), & \zeta\in \mathcal R_1, \\
            1/ \Phi(\zeta/\beta_\tau), & \zeta\in \mathcal R_2.
         \end{cases}
    \end{equation*}
    Since $\Phi_-(x)\Phi_+(x)=1$ on $(-1,1)$, we see that $\phi$ is globally defined and holomorphic on the whole Riemann surface $\mathcal R$. The transformation $z=\phi(\zeta)$ is a conformal bijection between  $\mathcal R$ and $\overline{\C}$, in such a way that the image of $\mathcal R_1$ (respectively, $\mathcal R_2$) is the exterior (respectively, interior) of the unit circle.
     It is convenient to observe that
     \begin{equation}
         \label{imagebyPhi}
         \phi(\ell) = \left[-A,-\frac{1}{A}\right]\cup \left[\frac{1}{A},A\right], \quad   A=A(\tau):=\Phi(1/\beta_\tau)=\frac{1}{\beta_\tau}+\frac{1}{\beta_\tau}\sqrt{1-\beta_\tau^2}\geq1. 
     \end{equation}
    
    The inverse $\phi^{-1}$ of $\phi $ is a rational function with a single pole at the origin:
    \begin{equation*}\label{varphi_inverse}
        \zeta=\phi^{-1}(z)=\frac{\beta_\tau}{2}\left(z+\frac{1}{z}\right).
    \end{equation*}
   We define $ F$ on $\mathbb{C}$ by 
    \begin{equation*}\label{F_Tilde_tau<-1}
        F(z)=\left( \mathcal F \circ \phi^{-1}\right) (z),\quad z\in\mathbb{C}.
    \end{equation*}
Taking into account \eqref{f_tau<-1_infinity_property} and that
    $$
    \phi^{-1} (z)= \frac{\beta_\tau}{2}z\left(1+o(1)\right), \quad z\to \infty; \qquad  \phi^{-1} (z)= \frac{\beta_\tau}{2z} \left(1+o(1)\right), \quad z\to 0, 
    $$
    we have that 
    \[
    F(z)
    =\frac{1+\tau}{z}\cdot\frac{2}{\beta_\tau}\left( 1+\mathcal O(z^{-1})\right), \quad z\to \infty, \]
    and
    \[
    F(z)
     = -(1+\tau)\cdot\frac{2z}{\beta_\tau}\left( 1+\mathcal O(z)\right),  \quad z \to 0.
     \]
    Thus, ${F}$ is analytic on $\mathbb{C}\setminus\phi(\ell)
    $, 
    and, by \eqref{bdrycondonR} and the definition of $F$,
    $$
    F_+(z)-F_-(z)=-\tau\pi i, \quad z\in \phi(\ell).
    $$
    Hence, we can recover $F$ from its boundary conditions using Sokhotski–-Plemelj's formula: 
\begin{equation*}\label{tau<-1_work_on_F_Tilde_Expression_1}
        \begin{split}
           F(z)&=\frac{1}{2\pi i}\int_{\phi(\ell)}\frac{\left( {F}_+- {F}_-\right)(t)}{t-z}\ dt =-\frac{\tau}{2 }\int_{\phi(\ell)}\frac{1}{t-z}\, dt. 
        \end{split}
    \end{equation*}
We can integrate it explicitly and select the appropriate branch of the logarithm by observing that $F(x)<0$ for $x>A$, with $A$ given in \eqref{imagebyPhi}. Thus,
$$
F(z)=-\frac{\tau}{2}\log\left(\frac{(z+1/A)(z-A)}{(z-1/A)(z+A)} \right).
$$
It is convenient to rewrite $F$ by introducing for $0< |a|\le 1$ the Blaschke factor
$$
b (z;a):=\frac{\left|a \right|}{a } \frac{a -z}{1-\bar{a}  z},
$$
so that
$$
F(z)=-\frac{\tau}{2}\log\left(\frac{b(z;-1/A) }{b(z;1/A)} \right).
$$
By definition,
$$
\mathcal F(\zeta)={F}(\phi(\zeta))=-\frac{\tau}{2}\log\left(\frac{b\left(\phi(\zeta);-1/A\right) }{b\left(\phi(\zeta);1/A\right)} \right),  \quad \zeta \in \mathcal R \setminus \ell. 
$$
Specializing $\zeta$ to $\mathcal R_1$ and using that $\mathcal F=f$ there, we get that
  \begin{equation}
  \label{f_tau<-1_expression_B}
  \begin{split}
 f(z)&= -\frac{\tau}{2}\log\left(\frac{(\Phi(z/\beta_\tau)+1/\Phi(1/\beta_\tau))(\Phi(z/\beta_\tau)-\Phi(1/\beta_\tau))}{(\Phi(z/\beta_\tau)-1/\Phi(1/\beta_\tau))(\Phi(z/\beta_\tau)+\Phi(1/\beta_\tau))} \right)
 \\
 &= -\frac{\tau}{2}\log\left(\frac{b\left(\Phi(z/\beta_\tau);-1/\Phi(1/\beta_\tau)\right) }{b\left(\Phi(z/\beta_\tau);1/\Phi(1/\beta_\tau)\right)} \right).
  \end{split}
\end{equation}
Using the expression for $A$ in \eqref{imagebyPhi}, we can write it explicitly in terms of $\beta_\tau$:
    \begin{equation}\label{f_tau<-1_expression_2}
        f(z)=\frac{\tau}{2}\log\left(  \frac{(z^2-\beta_\tau^2)^{1/2}+ \sqrt{1-\beta_\tau^2}}{(z^2-\beta_\tau^2)^{1/2}- \sqrt{1-\beta_\tau^2}}\right)= \frac{\tau}{2}\log\left(1+ \frac{2\sqrt{1-\beta_\tau^2}}{(z^2-\beta_\tau^2)^{1/2}- \sqrt{1-\beta_\tau^2}}\right),
    \end{equation}
which shows that, as $z\to\infty$,
    \[
    f(z)=\frac{\tau\sqrt{1-\beta_\tau^2}}{z}+\mathcal O(z^{-2}).
    \]
Matching this with the desired behavior \eqref{f_tau<-1_infinity_property} gives us the expression for $\beta_\tau$ in \eqref{supp_mu_tau<-1}. 

Notice that by the first expression in \eqref{f_tau<-1_expression_B}, for $x\in(\beta_\tau,1)$,
$$
f_+(x)-f_-(x)= -\frac{\tau}{2} \left(\log_+\left(\Phi(x/\beta_\tau)-\Phi(1/\beta_\tau)\right) - \log_-\left(\Phi(x/\beta_\tau)-\Phi(1/\beta_\tau)\right)  \right)=-\tau\pi i,
$$
which is the jump condition \eqref{f_tau<-1_limit_subtract_property} on  $ (\beta_\tau,1)$. The fact that the same jump condition holds on $(-1,-\beta_\tau)$ is proved in a similar fashion.

Using the second expression in \eqref{f_tau<-1_expression_B} and the property that the Blaschke product $b (z;a)$ satisfies $|b (z;a)|=1$ for $|z|=1$, we conclude that for $x\in [ -\beta_\tau,\beta_\tau ]$,  
$$
f_+(x)+f_-(x)= 2  \operatorname{Re}f_\pm(x) = -\tau \log\left| \frac{b\left(\Phi(x/\beta_\tau);-1/\Phi(1/\beta_\tau)\right) }{b\left(\Phi(x/\beta_\tau);1/\Phi(1/\beta_\tau)\right)} \right|=0,
$$
which gives us \eqref{f_tau<-1_limit_sum_property}.

    We now verify that \eqref{f_tau<-1_expression_2} satisfies \eqref{f_tau<-1_Re_f_sign_property}. By formula \eqref{f_tau<-1_expression_2},
    $$
       \Re f(z)=\frac{\tau}{2}\log\left|  \frac{(z^2-\beta_\tau^2)^{1/2}+ \sqrt{1-\beta_\tau^2}}{(z^2-\beta_\tau^2)^{1/2}- \sqrt{1-\beta_\tau^2}}\right|
    $$
    satisfies
    $$
      \Re f(z) = -   \Re f(-z),
    $$
    so it is sufficient to consider the case $x\in [\beta_\tau,1]$. Notice that for $0<a<1$ and $1<z<1/a$, 
    $$
    |b(z;-a)|- |b(z;a)|=-2 a\, \frac{  z^2-1 }{1-a^2
   z^2}<0,
    $$
    which, by \eqref{f_tau<-1_expression_B}, proves that 
    $$
    \Re f_\pm(x) = -\frac{\tau}{2}\log\left|\frac{b\left(\Phi(x/\beta_\tau);-1/A\right) }{b\left(\Phi(x/\beta_\tau);1/A\right)} \right|< 0, \quad x\in (\beta_\tau,1).
    $$

 By Sokhotski–-Plemelj's formula, and using \eqref{symmetryprinciple} and \eqref{f_tau<-1_expression_2}, we conclude that, for \newline$x\in (-\beta_\tau,\beta_\tau)$, 
    \begin{align*}
        \mu_\tau'(x)&=-\frac{1}{2\pi i}(C_+^{\mu_\tau}(x)-C_-^{\mu_\tau}(x))=-\frac{1}{2\pi i}(f_+(x)-f_-(x))-\tau\left(-\frac{1}{2\pi i}(C_+^\lambda(x)-C_-^\lambda(x))\right)\\
        &=-\frac{1}{\pi } \Im f_+(x)-\tau\lambda'(x) = -\frac{\tau}{2\pi }\arg\left(  \frac{i \sqrt{\beta_\tau^2-x^2} + \sqrt{1-\beta_\tau^2}}{i \sqrt{\beta_\tau^2-x^2}- \sqrt{1-\beta_\tau^2}}\right) -\frac{\tau}{2}\\
        &=-\frac{\tau}{\pi}\left(\frac{\pi}{2}-\arctan\left(\sqrt{\frac{1-\beta_\tau^2}{\beta_\tau^2-x^2}}\right)\right),
    \end{align*}
    and expression \eqref{mu_tau<-1_expression_1} follows. Since $\tau<0$ and $\arctan(\mathbb{R}_+)=(0,\frac{\pi}{2})$, we see that $\mu_\tau$ is a positive measure. 

    The analytic expression for $C^{\mu_\tau}(z)=f(z)-\tau C^\lambda(z)$ also follows from \eqref{f_tau<-1_expression_2}:
\begin{align*}
        C^{\mu_\tau}(z)&=\frac{\tau}{2}\log\left(  \frac{(z^2-\beta_\tau^2)^{1/2}+ \sqrt{1-\beta_\tau^2}}{(z^2-\beta_\tau^2)^{1/2}- \sqrt{1-\beta_\tau^2}}\right)-\frac{\tau}{2}\log\left(\frac{z+1}{z-1}\right) \\
    & = \tau\log\left(\frac{(z^2-\beta_\tau^2)^{1/2}+\sqrt{1-\beta_\tau^2}}{z+1}\right),
\end{align*}
which is the right hand side of \eqref{Cauchy_Transform_tau<-1_expression_2}.

We need to find the primitive (``$g$-function''),
$$
g^{\mu_\tau}(z)=\int^z C^{\mu_\tau}(y)\ dy
$$
of $C^{\mu_\tau}$, normalized in such a way that $g^{\mu_\tau}(z)=\log(z)+\mathcal{O}(z^{-1})$ as $z\to\infty$. Since for $\lambda$ such a primitive was calculated in \eqref{GforLebesgue}, we need to focus on a primitive of $f$, which can be found explicitly:
\begin{align*}
\int^z f(y)\, dy &=     z f(z)+ \frac{\tau}{2} \left( 2 \sqrt{1-\beta_\tau ^2} \log \left(\beta_\tau\Phi\left(\frac{z}{\beta_\tau}\right)\right)\right.
    \\
   & \quad\quad \left. -2\log
   \left(\frac{(z^2-\beta_\tau^2)^{1/2}+z\sqrt{1-\beta_\tau^2}}{ 
   \beta_\tau}\right)+\log
   \left(z^2-1\right)  \right).
\end{align*}

Combining it with \eqref{GforLebesgue}, we conclude that
\begin{align*}
    g^{\mu_\tau}(z)&=zC^{\mu_\tau}(z)+(1+\tau)\log\left(\beta_\tau\Phi\left(\frac{z}{\beta_\tau}\right)\right)-\tau\log\left(\frac{(z^2-\beta_\tau^2)^{1/2}+z\sqrt{1-\beta_\tau^2}}{\beta_\tau}\right)-\alpha,
\end{align*}
with
$$
\alpha=1+(1+\tau)\log(2)-\tau\log\left(\frac{1+\sqrt{1-\beta_\tau^2}}{\beta_\tau}\right),
$$
and expression \eqref{g_function_tau<-1} follows. Taking the negative real part of \eqref{g_function_tau<-1} establishes \eqref{log_potential_tau<-1} and, specialized to $x\in [-\beta_\tau, \beta_\tau]$, \eqref{Log_Potential_mu_when_-1<=tau_<=/(pi-2)_[-1,1]}. Since $0\in\supp(\mu_\tau)$, we obtain \eqref{equilibrium_constant_tau<-1} by evaluating $V^{\mu_\tau}(0)+\tau V^\lambda(0)$.
\end{proof}

\section{Disconnected support: \texorpdfstring{$\tau>2/(\pi-2)$}{\tau>2/(\pi-2)}} \label{sec:case3}

For $0\le k\le 1$ and $a>0$, consider the elliptic integrals
\begin{equation*}\label{complete_elliptic_integral_of_2nd_kind}
        E(k)=\int_0^{\pi/2}\sqrt{1-k^2\sin^2\theta}\, d\theta = \int_0^1 \sqrt{\frac{1-k^2 x^2}{1-x^2}}\, dx
    \end{equation*}
and
\begin{equation}\label{elliptic_integral_I}
        I(a,k)=\int_0^{\pi/2}\frac{1}{a+\sqrt{1-k^2\sin^2\theta}}\, d\theta = 
        \int_0^1 \frac{1}{a+\sqrt{1-k^2 x^2}} 
        \frac{dx}{\sqrt{1-x^2}}>0 .
    \end{equation}
It is well-known that $E(\cdot)$ is strictly decreasing on $[0,1]$ and $E:[0,1] \to [1,\pi/2]$ is a bijection. 

\begin{theorem}\label{theorem_case_tau>2/(pi-2)}
    For $\tau>\frac{2}{\pi-2}$, $\mu_\tau$ is an absolutely continuous measure on $\supp(\mu_\tau)=\newline\left[-1,-\beta_\tau\right]\cup\left[\beta_\tau,1\right]$, where $\beta_\tau\in (0,1)$ is the only solution of the equation 
    \begin{equation}\label{supp_mu_tau>2/(pi-2)}
       E(\beta_\tau)=   \frac{1}{\tau}+1  \in (1,\pi/2).
    \end{equation}
With the notation \eqref{elliptic_integral_I},
    \begin{equation}\label{mu_tau>2/(pi-2)}
    \begin{split}
         \mu_\tau'(x)&=\frac{\tau}{\pi}|x|\sqrt{\frac{x^2-\beta_\tau^2}{1-x^2}}\, I\left(\sqrt{1-x^2},\beta_\tau \right).
    \end{split}
    \end{equation}
Furthermore, for $z\in \C\setminus\supp(\mu_\tau)$,
\begin{equation}\label{Cauchy_Transform_mu_tau>2/(pi-2)}
        C^{\mu_\tau}(z)=\frac{\tau}{2}\left(\frac{z^2-\beta_\tau^2}{z^2-1}\right)^{1/2}\int_{-\beta_\tau}^{\beta_\tau} \sqrt{\frac{1-x^2}{\beta_\tau^2-x^2}}\frac{dx}{z-x} -\frac{\tau}{2}\log\left(\frac{z+1}{z-1}\right).
    \end{equation}
\end{theorem}
For the plot of the density of the equilibrium measure $\mu_\tau$ for $\tau=2$, see Figure~\ref{fig:density_tau>2/(pi-2)}.

\begin{figure}[ht!]
    \centering
    \begin{minipage}{.45\linewidth}
        \includegraphics[width=\linewidth]{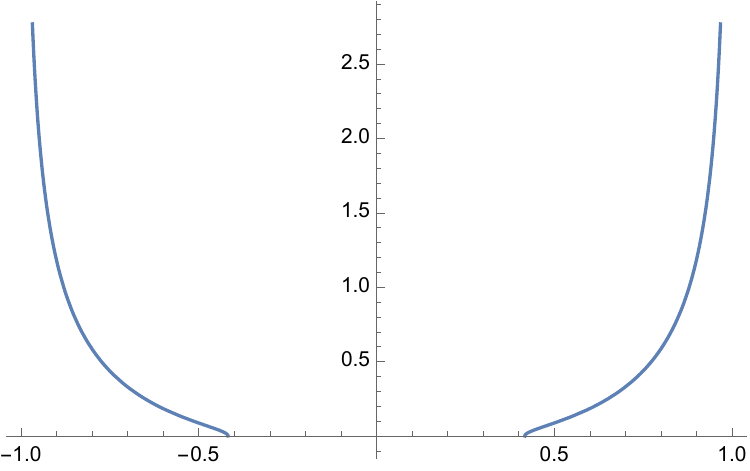}
        \caption{The density $\mu_{\tau}'$ for $\tau=2$, with $\beta_{2}\approx0.417299$.}
    \label{fig:density_tau>2/(pi-2)}
    \end{minipage}
    \hspace{.05\linewidth}
    \begin{minipage}{.45\linewidth}
        \includegraphics[width=\linewidth]{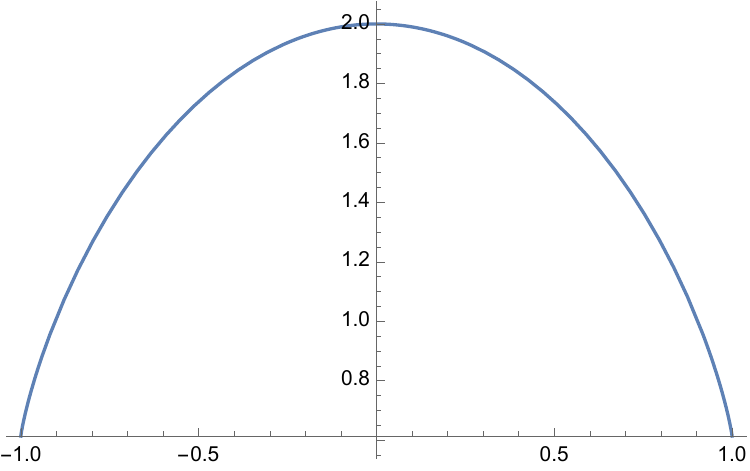}
        \caption{The external field for $\tau=2$.}
    \label{fig:external_field_tau=2}
    \end{minipage}
\end{figure}
\begin{remark}
\begin{enumerate}
\item[\textit{(i)}] We can calculate an approximation to the end point $\beta_\tau$ by inverting the well-known series expansion for $E$,
    $$
    E(k)=\frac{\pi}{2} \left(1-\sum_{n=1}^\infty \left( \frac{(2n-1)!!}{(2n)!!}\right)^2 \frac{k^{2n}}{2n-1}\right)=\frac{\pi}{2}\, {}_2F_1 \left(\begin{array}{c} 1/2, -1/2 \\ 1
    \end{array}; k^2\right),
    $$
    so that
    $$
    \beta_\tau= 2\sqrt{\alpha_\tau}\left(1-\frac{3}{8} \alpha_\tau-\frac{17}{128}\alpha_\tau^2+\dots \right), \quad \alpha_\tau:=\frac{2}{\pi} \left(\frac{\pi-2}{2}- \frac{1}{\tau}\right).
    $$
\item[\textit{(ii)}]  We can express $\mu_\tau$ and $C^{\mu_\tau}$ in terms of complete elliptic integrals. Letting
    \begin{equation}
        \label{defK}
    K(k)=\int_0^{\pi/2}\frac{1}{\sqrt{1-k^2\sin^2\theta}}\ d\theta=\int_0^1\frac{dx}{\sqrt{(1-x^2)(1-k^2x^2)}},
\end{equation}
    denote the complete elliptic integral of the first kind and
    \[\Pi(n,k)=\int_0^{\pi/2}\frac{1}{\left(1-n\sin^2\theta\right)\sqrt{1-k^2\sin^2\theta}}\ d\theta=\int_0^1\frac{1}{1-nx^2}\frac{dx}{\sqrt{(1-x^2)(1-k^2x^2)}}\]
    denote the complete elliptic integral of the third kind, we obtain
    \begin{align*}
         d\mu_\tau(x)&=\frac{\tau}{2\pi}\left[2|x|\sqrt{\frac{x^2-\beta_\tau^2}{1-x^2}}\left(K(\beta_\tau)+\frac{1-x^2}{x^2}\Pi\left(\frac{\beta_\tau^2}{x^2},\beta_\tau\right)\right)-\pi\right]\, dx
     \end{align*}
     and
     \begin{align*}
         C^{\mu_\tau}(z)&=\tau z\left(\frac{z^2-\beta_\tau^2}{z^2-1}\right)^{1/2}\left(K(\beta_\tau)+\frac{1-z^2}{z^2}\Pi\left(\frac{\beta_\tau^2}{z^2},     \beta_\tau\right)\right)-\frac{\tau}{2}\log\left(\frac{z+1}{z-1}\right).
     \end{align*}
\item[\textit{(iii)}] The density $\mu_\tau'$ exhibits the square-root asymptotics at the end-points of the support, corresponding to the standard ``soft'' (at $\pm \beta_\tau$) and ``hard edge'' (at $\pm 1$) behavior:
    \[\mu_\tau'(x)=\frac{\tau}{\pi}|x|K(\beta_\tau)\sqrt{\frac{x^2-\beta_\tau^2}{1-x^2}}\left(1+\mathcal{O}\left(1\right)\right),\quad x\to\pm1\]
    and
    \begin{align*}
        \mu_\tau'(x)=\frac{\tau\gamma_\tau(x)}{2\pi}\sqrt{|x\mp\beta_\tau|}\left(1+\mathcal{O}(x\mp\beta_\tau)\right),\quad x\to\pm\beta_\tau,
    \end{align*}
    where
\[
\gamma_\tau(x)=2|x|K(\beta_\tau)\sqrt{\frac{|x\pm \beta_\tau|}{1-x^2}}-\frac{\pi\beta_\tau^{3/2}\sqrt{1-\beta_\tau^2}}{\sqrt{2}}{}_2F_1 \left(\begin{array}{c} 3/2, 3/2 \\ 2
    \end{array}; \beta_\tau\right)=\mathcal{O}(1),\quad x\to\pm\beta_\tau.\]
    \end{enumerate}
\end{remark}
\begin{proof}
    When $\tau>\frac{2}{\pi-2}$, the external field becomes concave down and even, see Figure \ref{fig:external_field_tau=2}. Thus, we make an ansatz that $\operatorname{supp}(\mu_\tau)=[-1,-\beta_\tau]\cup[\beta_\tau,1]$, where the value of $\beta_\tau\in[0,1]$ is to be determined.
    As in Section~\ref{sec:case2}, define $f:\mathbb{C}\setminus[-1,1]\to\mathbb{C}$ by$$
    f(z):=C^{\mu_\tau}(z)+\tau C^\lambda(z)=\int\frac{ d{\mu_\tau}(x)}{z-x}+ \tau \int\frac{ d{\lambda}(x)}{z-x}.
    $$ 
    Our goal is to obtain an explicit expression for $f$ using the properties presented next, and in this way, to recover $\mu_\tau$. Again, $f$ is a holomorphic function on its domain such that
    \begin{align*}
        &f(z)=\frac{1+\tau}{z}+\mathcal O(z^{-2}),\hspace{0.44cm} z\to\infty;\\
        &\operatorname{Re}f_\pm(x)=0,\hspace{2.05cm}  x\in[-1,-\beta_\tau]\cup[\beta_\tau,1];\\
        &f_+(x)+f_-(x)=0,\hspace{1.2cm} x\in(-1,-\beta_\tau)\cup(\beta_\tau,1);\quad\text{and}\\
        &f_+(x)-f_-(x)=-\tau\pi i,\hspace{0.52cm} x\in(-\beta_\tau,\beta_\tau),
    \end{align*}
    and we make the ansatz
    \begin{equation*}
        \operatorname{Re} f_\pm(x)\geq0 \text{ on } (0,\beta_\tau),\quad \text{and} \quad \operatorname{Re} f_\pm(x)\leq0 \text{ on } (-\beta_\tau,0).\label{f_tau>2/(pi-2)_Re_f_sign_property}
    \end{equation*}
    Since $V^{\mu_\tau}(x)+\tau V^\lambda(x)$ is even on $[-\beta_\tau,\beta_\tau]$, $\operatorname{Re}f_\pm(x)$ is odd on $[-\beta_\tau,\beta_\tau]$, so
    \begin{equation*}\label{f_tau>2/(pi-2)_integral=0_property}
        \operatorname{Re}\int_{-\beta_\tau}^{\beta_\tau} f_\pm(x)\ dx=0.
    \end{equation*}
    \indent Function $h:\mathbb{C}\setminus([-1,-\beta_\tau]\cup[\beta_\tau,1])\to\mathbb{C}$, defined by
    \begin{equation*}
        h(z)=\left(\frac{z^2-1}{z^2-\beta_\tau^2}\right)^{1/2},
    \end{equation*}
    is holomorphic and even on its domain and satisfies the following properties:
    \begin{align}
        &h_-(x)=-h_+(x),\quad\ x\in(-1,-\beta_\tau)\cup(\beta_\tau,1);\nonumber\\
        &h_\pm(x)\in \pm i\mathbb{R}_{+},\hspace{0.95cm} x\in(\beta_\tau,1);\nonumber\\
        &h_\pm(x)\in \mp i\mathbb{R}_{+},\hspace{0.95cm}  x\in(-1,-\beta_\tau);\nonumber\\
        &\lim_{z\to \infty}h(z)=1;\hspace{0.65cm}\quad \text{and}\nonumber\\
        &h(x)>0,\hspace{1.86cm} x\in(-\beta_\tau,\beta_\tau).\label{h_positive_on_(-beta,beta)}
    \end{align}
    Define $F:\mathbb{C}\setminus([-1,-\beta_\tau]\cup[\beta_\tau,1])\to\mathbb{C}$ by
    \begin{equation*}\label{F_tau>2/(pi-2)_definition}
        F(z)=h(z)f(z).
    \end{equation*}
    Then $F$ is holomorphic on its domain, and by the above properties of $f$ and $h$, it satisfies
    \begin{align}
        &F_+(x)-F_-(x)=-\tau\pi ih(x),\quad x\in(-\beta_\tau,\beta_\tau);\label{F_(1)_tau>2/(pi-2)}\\
        &F_+(x)-F_-(x)=0,\hspace{1.815cm} x\in(-1,-\beta_\tau)\cup(\beta_\tau,1);\label{F_(2)_tau>2/(pi-2)}\\
        &F_\pm(x)\in\mathbb{R},\hspace{3.33cm} x\in[-1,-\beta_\tau]\cup[\beta_\tau,1];\label{F_(2')_tau>2/(pi-2)}\\
        &F(z)=\frac{1+\tau}{z}+\mathcal O(z^{-2}),\hspace{1.12cm} z\to\infty;\label{F_(3)_tau>2/(pi-2)}\\
        &\operatorname{Re}\int_{-\beta_\tau}^{\beta_\tau} \frac{F_\pm(x)}{h(x)}\ dx=0\label{F_(4)_tau>2/(pi-2)};\\
        &\operatorname{Re}F_\pm(x)\geq0  \text{ on } (0,\beta_\tau);\quad \text{and} \quad  \operatorname{Re}F_\pm(x)\leq0 \text{ on }  x\in(-\beta_\tau,0).\label{F_(5)_tau>2/(pi-2)}
    \end{align}

The Sokhotski–-Plemelj formula implies that, if $g$ is a smooth function on $(-\beta_\tau, \beta_\tau)$, then
\begin{equation}
    \label{SPformula}
    \text{p.v.}\, \int_{-\beta_\tau}^{\beta_\tau} \frac{g(y)}{x-y}\, dy = \lim_{\varepsilon\downarrow 0}\left(\int_{-\beta_\tau}^{\beta_\tau} \frac{g(y)}{x\pm i\varepsilon-y}\, dy\right) \pm \pi i g(x), \quad x\in ( -\beta_\tau, \beta_\tau).
\end{equation}
  By \eqref{F_(1)_tau>2/(pi-2)}, \eqref{F_(2)_tau>2/(pi-2)}, and \eqref{SPformula}, we have
    \begin{equation}\label{F_tau>2/(pi-2)}
        F(z)=\frac{1}{2\pi i}\int_{-\beta_\tau}^{\beta_\tau}\frac{-\tau\pi ih(x)}{x-z}\ dx=-\frac{\tau}{2}\int_{-\beta_\tau}^{\beta_\tau}\frac{h(x)}{x-z}\ dx.
    \end{equation}
    \indent We now verify that \eqref{F_tau>2/(pi-2)} satisfies properties \eqref{F_(2')_tau>2/(pi-2)}-\eqref{F_(5)_tau>2/(pi-2)}. \eqref{F_(2')_tau>2/(pi-2)} is satisfied because of \eqref{h_positive_on_(-beta,beta)}. We obtain
    \begin{equation*}\label{f_tau>2/(pi-2)_Infinity_Property_Work_1}
    \begin{split}
        F(z)&=\frac{\tau E(\beta_\tau)}{z}\left(1+\mathcal{O}\left(\frac{1}{z}\right)\right),\quad z\to\infty.
    \end{split}
    \end{equation*}
    Matching this with the desired behavior \eqref{F_(3)_tau>2/(pi-2)} gives us the expression for $\beta_\tau$ in \eqref{supp_mu_tau>2/(pi-2)}.\newline
    \indent By \eqref{SPformula}, we have for $x\in(-\beta_\tau,\beta_\tau)$ that
    \begin{align*}
        F_\pm(x)=\mp\frac{\tau\pi i}{2}h(x)+\text{p.v.}\ \frac{\tau}{2}\int_{-\beta_\tau}^{\beta_\tau}\frac{h(y)}{x-y}\ dy.
    \end{align*}
    Thus, \eqref{h_positive_on_(-beta,beta)} implies for all $x\in(-\beta_\tau,\beta_\tau)$ that
\begin{equation}\label{Re(F_pm)_principal_value_expression}
        \Re F_\pm(x)=\operatorname{p.v.}\frac{\tau}{2}\int_{-\beta_\tau}^{\beta_\tau}\frac{h(y)}{x-y}\ dy.
    \end{equation}
    Observing that
    \eqref{Re(F_pm)_principal_value_expression} is odd on $(-\beta_\tau,\beta_\tau)$ and $h$ is even on $(-\beta_\tau,\beta_\tau),$ we obtain \eqref{F_(4)_tau>2/(pi-2)}. 
    To prove that \eqref{F_tau>2/(pi-2)} satisfies \eqref{F_(5)_tau>2/(pi-2)}, it suffices to show that
    \begin{equation*}\label{Int_Ineq_Sufficient_1b}
       \operatorname{p.v.}\int_{-\beta_\tau}^{\beta_\tau}\frac{h(y)}{x-y}\ dy
    \end{equation*}
    is $\leq 0$ for $x\in(-\beta_\tau,0)$, and $\geq 0$ for $x\in(0,\beta_\tau)$. This is an immediate consequence of the following identity:
    \begin{lem} \label{lem_identity}
    For $x\in (-\beta_\tau,\beta_\tau)$,
    $$
   \operatorname{p.v.} \int_{-\beta_\tau}^{\beta_\tau} \sqrt{\frac{1-y^2}{\beta_\tau^2-y^2}}\, \frac{dy}{x-y}= x \int_{-\beta_\tau}^{\beta_\tau}  \frac{dy}{\sqrt{\beta_\tau^2-y^2}\, \left(\sqrt{1-x^2} + \sqrt{1-y^2}   \right) } =2xI\left(\sqrt{1-x^2},\beta_\tau \right). 
    $$
    \end{lem}
Indeed, since
\begin{equation}\label{Int_Identity}
\int_{-\beta_\tau}^{\beta_\tau} \frac{dy}{\sqrt{\beta_\tau^2-y^2}(z-y)}\, dy=\frac{\pi}{\sqrt{z^2-\beta_\tau^2}}, \quad z\in \C\setminus[-\beta_\tau,\beta_\tau],
\end{equation}
using \eqref{SPformula} we conclude that
$$
\text{p.v.}\, \int_{-\beta_\tau}^{\beta_\tau} \frac{1}{\sqrt{\beta_\tau^2-y^2}}\frac{dy}{x-y}  =0.
$$
Hence,
\begin{align*}
    \operatorname{p.v.} \int_{-\beta_\tau}^{\beta_\tau} \sqrt{\frac{1-y^2}{\beta_\tau^2-y^2}}\, \frac{dy}{x-y} &= \operatorname{p.v.} \int_{-\beta_\tau}^{\beta_\tau} \frac{\sqrt{1-y^2}-\sqrt{1-x^2}}{\sqrt{\beta_\tau^2-y^2}}\, \frac{dy}{x-y}\\
    &= \int_{-\beta_\tau}^{\beta_\tau}  \frac{x+y}{\sqrt{\beta_\tau^2-y^2}\, \left(\sqrt{1-x^2} + \sqrt{1-y^2}   \right) }\, dy\\
    &=x \int_{-\beta_\tau}^{\beta_\tau}  \frac{dy}{\sqrt{\beta_\tau^2-y^2}\, \left(\sqrt{1-x^2} + \sqrt{1-y^2}   \right) },
\end{align*}
which proves the identity.

Hence, properties \eqref{F_(1)_tau>2/(pi-2)}--\eqref{F_(5)_tau>2/(pi-2)} of $F$ have been established. Unraveling the previous transformation, we have that 
    \begin{equation}\label{f_int_formula_tau>2/(pi-2)}
        f(z)=\frac{F(z)}{h(z)}=\frac{\tau}{2}\left(\frac{z^2-\beta_\tau^2}{z^2-1}\right)^{1/2}\int_{-\beta_\tau}^{\beta_\tau}\sqrt{\frac{1-x^2}{\beta_\tau^2-x^2}}\frac{dx}{z-x},
    \end{equation}
    which implies that
    \begin{equation*}
        C^{\mu_\tau}(z)=f(z)-\tau C^\lambda(z)= \frac{\tau}{2}\left(\frac{z^2-\beta_\tau^2}{z^2-1}\right)^{1/2}\int_{-\beta_\tau}^{\beta_\tau} \sqrt{\frac{1-x^2}{\beta_\tau^2-x^2}}\frac{dx}{z-x} -\frac{\tau}{2}\log\left(\frac{z+1}{z-1}\right),
    \end{equation*}
    see \eqref{Cauchy_Transform_mu_tau>2/(pi-2)}.
    
    By applying the Stieltjes-Perron inversion formula and then \eqref{Int_Identity}, we conclude for \newline$x\in [-1,-\beta_\tau]\cup[\beta_\tau,1]$ that
    \begin{align*}
        \mu_\tau'(x)&=-\frac{1}{\pi}\operatorname{Im}C_+^{\mu_\tau}(x)=-\frac{1}{\pi}\operatorname{Im}\left(f_+(x)-\tau C_+^{\lambda}(x)\right)\\
        &=\frac{\tau}{2\pi}\left[\operatorname{sgn}(x)\sqrt{\frac{x^2-\beta_\tau^2}{1-x^2}}\int_{-\beta_\tau}^{\beta_\tau}\sqrt{\frac{1-y^2}{\beta_\tau^2-y^2}}\frac{dy}{x-y}-\pi\right]\\
        &=\frac{\tau}{2\pi}\left[\operatorname{sgn}(x)\sqrt{\frac{x^2-\beta_\tau^2}{1-x^2}}\int_{-\beta_\tau}^{\beta_\tau}\sqrt{\frac{1-y^2}{\beta_\tau^2-y^2}}\frac{dy}{x-y}-\int_{-\beta_\tau}^{\beta_\tau}\frac{\sqrt{x^2-\beta_\tau^2}}{\sqrt{\beta^2_\tau-y^2}}\frac{dy}{x-y}\right]\\
        &=\frac{\tau}{2\pi}|x|\sqrt{\frac{x^2-\beta_\tau^2}{1-x^2}}\int_{-\beta_\tau}^{\beta_\tau}\frac{1}{\sqrt{1-y^2}+\sqrt{1-x^2}}\frac{dy}{\sqrt{\beta_\tau^2-y^2}} \\
        &= \frac{\tau}{ \pi}|x|\sqrt{\frac{x^2-\beta_\tau^2}{1-x^2}}\, I\left(\sqrt{1-x^2},\beta_\tau \right)\ge 0.
    \end{align*}
This establishes \eqref{mu_tau>2/(pi-2)}.
\end{proof}

\begin{corollary}
 For $\tau>\frac{2}{\pi-2}$, the equilibrium constant $\omega_\tau$ in \eqref{equilibrium} is given by
\begin{equation}\label{equilibrium_constant_tau>2/(pi-2)}
 \omega_\tau    = \int_1^{+\infty} \left(\frac{\tau}{2}\sqrt{\frac{x^2-\beta_\tau^2}{x^2-1}}\left( \int_{-\beta_\tau}^{\beta_\tau}\sqrt{\frac{1-s^2}{\beta_\tau^2-s^2}}\frac{ds}{x-s} \right)-\frac{1+\tau}{x}\right)\,dx.
    \end{equation}
It can be expressed in terms of a convergent series
\begin{equation}
    \label{powerseriesomega}
\omega_\tau=  \frac{1+\tau}{2}\, \log \left(\frac{4}{1-\beta_\tau^2}
   \left(\frac{1-\beta_\tau}{1+\beta_\tau}\right)^{\beta_\tau}\right) +  \tau \sum_{k=1}^{+\infty}    c_{2k-1}   \, c_{2k} \, \beta_\tau^{2k} ,
\end{equation}
where 
\begin{equation}
    \label{c0}
c_0=E(\beta_\tau)=\frac{1+\tau}{\tau}
\end{equation}
and for $k\in \N$,
\begin{equation}
    \label{defCn}
c_k:=\int_{0}^{1}\sqrt{\frac{1-\beta_\tau^2s^2}{1-s^2}}s^k ds  =\frac{\sqrt{\pi} \Gamma\left(\frac{k+1}{2}\right)}{2 \Gamma\left(\frac{k}{2}+1\right)}\, { }_2 F_1\left(-\frac{1}{2}, \frac{k+1}{2} ; \frac{k}{2}+1 ; \beta_\tau^2\right) .
\end{equation}
The coefficients $c_k$ can also be computed from the three-term recurrence relations
\begin{equation}
    \label{rec2}
    \begin{split}
      & 4 \beta_\tau^2\, c_3-\left(1+3 \beta_\tau^2\right) c_1+1=0  , \\
&\beta_\tau^2(k+3) c_{k+2}-\left[k+\beta_\tau^2(k+2)\right] c_k+(k-1) c_{k-2}=0, \quad k \geq 2 ,
    \end{split}
\end{equation}
with the initial values $c_0$ in \eqref{c0}, 
$$
c_1= \frac{1}{2}+\frac{1-\beta_\tau^2}{2 \beta_\tau} \log\left(\frac{1+\beta_\tau}{1-\beta_\tau} \right) , \quad \text{and} \quad c_2= \frac{(2\beta_\tau^2-1)(1+\tau)}{3 \beta_\tau^2 \tau}  +\frac{1-\beta_\tau^2}{3 \beta_\tau^2}K\left( \beta_\tau\right),
$$
with $K$ defined in \eqref{defK}.
\end{corollary}
\begin{remark}
\begin{enumerate}
\item[\textit{(i)}] 
We will show below that the equilibrium constant $\omega_\tau$ is also continuous across $\tau=2/(\pi-2)$:
$$
\lim_{\tau \to (2/(\pi-2))^-} \omega_\tau = \lim_{\tau \to (2/(\pi-2))^+} \omega_\tau=\frac{\pi}{\pi-2}\log(2).
$$

\item[\textit{(ii)}] The three-term recurrence \eqref{rec2} can be obtained from the contiguous relation of the ${ }_2 F_1$ in \eqref{defCn} under $k \rightarrow k \pm 2$.
\end{enumerate}
\end{remark}
\begin{proof}
With $
f(z)=C^{\mu +\tau \lambda}(z)$, 
$$
V^\mu(z)+\tau V\lambda(z)=V^{\mu+\tau \lambda}(z)=V^{\mu+\tau \lambda}(1)-\Re \int_1^z f(x)\,dx.
$$
Since $1\in\supp(\mu_\tau)$, we have that $V^{\mu+\tau \lambda}(1)=\omega_\tau$. Additionally, $V^{\mu+\tau \lambda}$ has the property that
$$
V^{\mu+\tau \lambda}(z)=-(1+\tau) \log|z|+ \mathcal O\left(\frac{1}{z} \right), \quad z\to \infty.
$$
Thus,
\begin{align*}
 0&=   \lim_{z\to \infty} \left( V^{\mu+\tau \lambda}(z)+(1+\tau) \log|z|\right)= \lim_{z\to \infty} \left( \omega_\tau -\Re \int_1^z f(x)\,dx +(1+\tau) \log|z|\right) \\
 &= \lim_{z\to \infty} \left( \omega_\tau -\Re \int_1^z \left( f(x)-\frac{1+\tau}{x}\right)\,dx  \right),
\end{align*}
which gives us the identity
\begin{align*}
\omega_\tau&=  \int_1^{+\infty} \left( f(x)-\frac{1+\tau}{x}\right)\,dx.
\end{align*}
It remains to use the expression for $f$ in  \eqref{f_int_formula_tau>2/(pi-2)} to arrive at \eqref{equilibrium_constant_tau>2/(pi-2)}. 

To show that this is a converging improper integral, we use \eqref{supp_mu_tau>2/(pi-2)}:
\begin{equation*}
  \frac{\tau}{2}   \int_{-\beta_\tau}^{\beta_\tau}\sqrt{\frac{1-s^2}{\beta_\tau^2-s^2}}\, ds = \tau E(\beta_\tau)= 1+\tau.
\end{equation*}
Thus,
$$
\frac{\tau}{2}   \int_{-\beta_\tau}^{\beta_\tau}\sqrt{\frac{1-s^2}{\beta_\tau^2-s^2}}\frac{ds}{x-s} - \frac{1+\tau}{x}= \frac{\tau}{2x}   \int_{-\beta_\tau}^{\beta_\tau}\sqrt{\frac{1-s^2}{\beta_\tau^2-s^2}}\frac{s\ ds}{x-s},
$$
and we can write
$$
\omega_\tau  = W_1+W_2,
$$
where
\begin{align*}
W_1 & = \frac{\tau}{2}\int_1^{+\infty} \left(\sqrt{\frac{x^2-\beta_\tau^2}{x^2-1}} -1\right)\left( \int_{-\beta_\tau}^{\beta_\tau}\sqrt{\frac{1-s^2}{\beta_\tau^2-s^2}}\frac{ds}{x-s} \right) \,dx\\
&= \frac{\tau(1-\beta_\tau^2)}{2}\int_1^{+\infty} \frac{1}{\sqrt{ x^2-1} \left( \sqrt{ x^2-1}+\sqrt{ x^2-\beta_\tau^2}\right) } \left( \int_{-\beta_\tau}^{\beta_\tau}\sqrt{\frac{1-s^2}{\beta_\tau^2-s^2}}\frac{ds}{x-s} \right) \,dx
\end{align*}
and
\begin{align*}
W_2 & = \int_1^{+\infty} \left(\frac{\tau}{2}  \left( \int_{-\beta_\tau}^{\beta_\tau}\sqrt{\frac{1-s^2}{\beta_\tau^2-s^2}}\frac{ds}{x-s} \right)-\frac{1+\tau}{x}\right)\,dx = \frac{\tau }{2}\int_1^{+\infty}   \left( \int_{-\beta_\tau}^{\beta_\tau}\sqrt{\frac{1-s^2}{\beta_\tau^2-s^2}}\frac{s \, ds}{x-s} \right) \,\frac{dx}{x}.
\end{align*}
Both improper integrals are clearly converging. 

To study the behavior of $W_j$, $j=1,2$, as $\tau\to 2/(\pi-2)^+$, $\beta_\tau\to 0^+$, we can make a change of variables in the interior integrals and get alternative expressions: 
\begin{equation} \label{I1}
W_1   = \frac{\tau(1-\beta_\tau^2)}{2}\int_1^{+\infty} \frac{1}{\sqrt{ x^2-1} \left( \sqrt{ x^2-1}+\sqrt{ x^2-\beta_\tau^2}\right) } \left( \int_{-1}^{1}\sqrt{\frac{1-\beta_\tau^2 s^2}{1-s^2}}\frac{ds}{x-\beta_\tau s} \right) \,dx
\end{equation}
and
\begin{equation} \label{I2}
W_2   = \frac{\tau \beta_\tau }{2}\int_1^{+\infty}   \left( \int_{-1}^{1}\sqrt{\frac{1-\beta_\tau^2s^2}{1-s^2}}\frac{s\, ds}{x-\beta_\tau s} \right) \,\frac{dx}{x}.
\end{equation}
As $\tau\to 2/(\pi-2)^+$, $\beta_\tau\to 0^+$, and
\begin{align*}
\lim_{\tau\to 2/(\pi-2)^+} W_1  = & \frac{2 }{2(\pi-2)}\int_1^{+\infty} \frac{1}{x \sqrt{ x^2-1} \left(x+ \sqrt{ x^2-1}\right) } \left( \int_{-1}^{1}\frac{ds}{\sqrt{1-s^2}}   \right) \,dx \\
&= \frac{\pi}{\pi-2}\int_1^{+\infty} \frac{1}{x \sqrt{ x^2-1} \left(x+ \sqrt{ x^2-1}\right) }  \,dx  = \frac{\pi}{\pi-2}\, \log(2),
\end{align*}
while by \eqref{I2}, $\lim_{\tau\to 2/(\pi-2)^+} W_2= 0$. Combining it with \eqref{equilibrium_constant_-1<=tau<=2/(pi-2)} we conclude that $\omega_\tau$ is continuous at $\tau=2/(\pi-2)$.

Finally, we address a series representation of $\omega_\tau$ for $\tau>2/(\pi-2)$. Since for $x\ge 1$, $|\beta_\tau s|\le \beta_\tau<1$, we can expand $1/(x-\beta_\tau s)$ into geometric series, and the definition \eqref{defCn} yields  
\begin{equation}
    \label{seriesexp1}
\int_{-1}^{1}\sqrt{\frac{1-\beta_\tau^2s^2}{1-s^2}}\frac{ds}{x-\beta_\tau s} =\sum_{k=0}^{+\infty} \frac{\beta_\tau^k}{x^{k+1}}\int_{-1}^{1}\sqrt{\frac{1-\beta_\tau^2s^2}{1-s^2}}s^k \,ds  = 2 \sum_{k=0}^{+\infty} \frac{  c_{2k} \beta_\tau^{2k}}{x^{2k+1}}  
\end{equation}
and
\begin{equation}
    \label{seriesexp2}
 \int_{-1}^{1}\sqrt{\frac{1-\beta_\tau^2s^2}{1-s^2}}\frac{s\, ds}{x-\beta_\tau s} =\frac{1}{\beta_\tau^k}\left( \int_{-1}^{1}\sqrt{\frac{1-\beta_\tau^2s^2}{1-s^2}}\frac{x\, ds}{x-\beta_\tau s} -2 c_0 \right)  =2 \sum_{k=1}^{+\infty} \frac{  c_{2k} \beta_\tau^{2k-1}}{x^{2k}}  .
\end{equation}
Replacing \eqref{seriesexp1} into \eqref{I1}, we get
$$
W_1= \tau(1-\beta_\tau^2) \sum_{k=0}^{+\infty}    c_{2k} \beta_\tau^{2k} \int_1^{+\infty} \frac{1}{x^{2k+1}   \left(   x^2-1 +\sqrt{ ( x^2-1)(x^2-\beta_\tau^2})\right) }   \,dx .
$$
Notice that for $k\ge 1$,
\begin{align*}
& \int_1^{+\infty} \frac{1}{x^{2k+1}   \left(   x^2-1 +\sqrt{ ( x^2-1)(x^2-\beta_\tau^2})\right) }   \,dx \\
&= \frac{1}{1-\beta_\tau^2} \left( \int_1^{+\infty} \sqrt{\frac{x^2-\beta_\tau^2}{x^2-1}} \frac{   dx }{x^{2k+1}   }   \,dx   - \frac{  1}{2k } \right)\\
&= \frac{1}{1-\beta_\tau^2} \left( \int_0^1 \sqrt{\frac{1-\beta_\tau^2 s^2}{1-s^2}} s^{2k-1}ds   - \frac{  1}{2k } \right) = \frac{1}{1-\beta_\tau^2} \left( c_{2k-1}    - \frac{  1}{2k } \right),
\end{align*}
while
$$
\int_1^{+\infty} \frac{1}{x    \left(   x^2-1 +\sqrt{ ( x^2-1)(x^2-\beta_\tau^2})\right) }   \,dx = \frac{1}{2
   \left(1-\beta_\tau^2\right)} \log \left(\frac{4}{1-\beta_\tau^2}
   \left(\frac{1-\beta_\tau}{1+\beta_\tau}\right)^{\beta_\tau}\right). 
$$
thus,
$$
W_1= \frac{1+\tau}{2}    \log \left(\frac{4}{1-\beta_\tau^2}
   \left(\frac{1-\beta_\tau}{1+\beta_\tau}\right)^{\beta_\tau}\right)+ \tau  \sum_{k=1}^{+\infty}    c_{2k} \beta_\tau^{2k} \left( c_{2k-1}    - \frac{  1}{2k } \right).
$$

Similarly, using \eqref{seriesexp2}  in \eqref{I2} we get
$$
\quad W_2  =  \tau    \sum_{k=1}^{+\infty}  \frac{c_{2k}}{2k}  \, \beta_\tau^{2k}  .
$$
A combination of these two expressions yields
\begin{align*}
\omega_\tau & =  \frac{1+\tau}{2}    \log \left(\frac{4}{1-\beta_\tau^2}
   \left(\frac{1-\beta_\tau}{1+\beta_\tau}\right)^{\beta_\tau}\right)+ \tau  \sum_{k=1}^{+\infty}    c_{2k} \beta_\tau^{2k} \left( c_{2k-1}    - \frac{  1}{2k } \right)+ \tau    \sum_{k=1}^{+\infty}  \frac{c_{2k}}{2k}  \, \beta_\tau^{2k} ,
\end{align*}
that is, \eqref{powerseriesomega}.

Finally, define
$$
d_k:=\int_0^1 \frac{s^k \, ds}{\sqrt{(1-s^2)(1-\beta_\tau^2 s^2)}} \,d \theta.
$$ 
Then
\begin{equation}
    \label{identitycandJ}
c_k=d_k-\beta_\tau^2 d_{k+2} .
\end{equation}
Furthermore, for 
$$
h_k(s):= s^{k-1} (1-s^2)^{1/2}(1-\beta_\tau^2s^2)^{1/2}
$$ 
we have
$$
h_k'(s)= \sqrt{\frac{1-\beta_\tau^2 s^2}{1-s^2}}  \left[(k-1) s^{k-2} -k s^k \right] -\frac{\beta_\tau^2 s^{k}(1-s^2)}{\sqrt{(1-s^2)(1-\beta_\tau^2 s^2)}} .
$$
Since for $k \geq 2$, $h_k(0)= h_k(1)=0$, integrating $h_k'$ over $[0, 1]$ for $k\geq2$ gives
\begin{equation*}
0=(k-1) \, c_{k-2}-k \, c_k-\beta_\tau^2 d_k+\beta_\tau^2 d_{k+2} .
\end{equation*}
We can use \eqref{identitycandJ} to replace $\beta_\tau^2 d_{k+2}=d_k-c_k$:
$$
0=(k-1) \, c_{k-2}-(k+1) \, c_k+(1-\beta_\tau^2) d_k  .
$$
In particular, 
$$
(1-\beta_\tau^2) d_k=-(k-1) \, c_{k-2}+(k+1) \, c_k.
$$
Replacing this equality in \eqref{identitycandJ}, with $k$ and $k+2$, yields
$$
(1-\beta_\tau^2) c_k=-(k-1) \, c_{k-2}+(k+1) \, c_k-\beta_\tau^2 \left(-(k+1) \, c_{k}+(k+3) \, c_{k+2}\right) ,
$$
which is equivalent to recurrence \eqref{rec2} for $k\ge 2$.

On the other hand, for $k=1$, the boundary term $h_1(0)=1 \neq 0$ modifies the equation, and after integration we get
$$
4 \beta_\tau^2\, c_3=\left(1+3 \beta_\tau^2\right) c_1-1 ,
$$
as in \eqref{rec2}. 

Finally, the expressions for $c_1$ and $c_2$ are obtained by direct computation, while the second expression in \eqref{defCn} is a consequence of the Euler (or Gauss) integral representation of the hypergeometric function, see \cite[formula (15.6.1)]{DLMF}.
\end{proof}

\section*{Acknowledgments}

    The second author was partially supported by the Spanish project PID2021-124472NB-I00, funded by MICIU/AEI/10.13039/501100011033 and by ``ERDF A way of making Europe.'' He also acknowledges the support of Junta de Andaluc\'{\i}a (research group FQM-229 and Instituto Interuniversitario Carlos I de F\'{\i}sica Te\'orica y Computacional). 

    We also benefited from fruitful discussions with Guilherme Silva.


\def\cprime{$'$}

\end{document}